\documentclass[11pt,letterpaper]{amsart}
\usepackage{amsthm,amsmath,amsfonts,latexsym,amssymb,mathrsfs,color,changepage,graphicx}
\usepackage{amsfonts}
\usepackage{amssymb}
\usepackage{graphicx}
\usepackage{pstricks}
\usepackage{amsmath}
\usepackage{amsxtra}
\usepackage{bm}
\usepackage{enumitem}
\usepackage{marginnote}
\usepackage[normalem]{ulem}
\usepackage{multicol}

\makeatletter
\newcommand* \bigcdot{\mathpalette \bigcdot@{.5}}
\newcommand* \bigcdot@[2]{\mathbin{\vcenter{\hbox{\scalebox{#2}{$\m@th#1\bullet$}}}}}
\makeatother

\makeatletter
\newcommand\appendix@section[1]{%
 \refstepcounter{section}%
 \orig@section*{Appendix \@Alph\c@section: #1}%
 \addcontentsline{toc}{section}{Appendix \@Alph\c@section: #1}%
}
\let\orig@section\section
\g@addto@macro\appendix{\let\section\appendix@section} \makeatother

\theoremstyle{definition}
\newtheorem{Def}{Definition}[section]
\newtheorem{df}[Def]{Definition}
\newtheorem{thm}[Def]{Theorem}
\newtheorem{cor}[Def]{Corollary}

\newtheorem{lem}[Def]{Lemma}

\newtheorem{rem}[Def]{Remark}
\newtheorem{ex}[Def]{Example}

\renewcommand{\thefootnote}{\fnsymbol{footnote}}

\begin{document}

\title
[Circle companions of Hardy spaces of the unit disk]
{Circle companions of \\ Hardy spaces of the unit disk
}

\author{Ra{\'u}l\ E.\ Curto, In Sung Hwang, Sumin Kim and Woo Young Lee}

\date{}

\maketitle

%
%
%
%

\noindent{\bf Abstract.} This paper gives a complete answer to the
following problem: Find the circle companion of the Hardy space of the
unit disk with values in the space of all bounded linear operators
between two separable Hilbert spaces. \ Classically, the problem
asks whether for each function $h$ on the unit {\it disk}, there
exists a ``boundary function" $bh$ on the unit {\it circle} such
that the mapping $bh\mapsto h$ is an isometric isomorphism between
Hardy spaces of the unit circle and the unit disk with values in
some Banach space. \ For the case of bounded linear operator-valued
functions, we construct a Hardy space of the unit circle such that
its elements are SOT measurable, and their norms are integrable:
indeed, this new space is isometrically isomorphic to the Hardy
space of the unit disk via a ``strong Poisson integral."

\medskip


\setcounter{page}{1}


\renewcommand{\thefootnote}{}

%
%
%
%

\section{Introduction}

We solve an old and outstanding problem in the theory
of Hardy spaces. \ For $1 \le p \le \infty$ and $X$ a Banach space,
consider the Hardy space $H^p(\mathbb D, X)$ of $X$--valued
functions defined on the unit disk $\mathbb D$. \ For each $h \in
H^p(\mathbb D, X)$, we try to associate a function $bh$, which
captures the boundary values of $h$. \ Our goal is to
identify a Banach space $\mathcal{C}$ of $X$--valued functions defined on the unit
circle $\mathbb{T}$ which represent, in a natural and canonical way,
the boundary values of functions in $H^p(\mathbb D, X)$.  \ When the
mapping $h\mapsto bh$ is an isometric isomorphism from $H^p(\mathbb
D, X)$ onto $\mathcal{C}$, we say that $\mathcal{C}$ is the
``circle companion" of $H^p(\mathbb D, X)$.

In this paper, we find the circle companion of the Hardy space of the
unit disk with values in $\mathcal{B}(D,E)$, the space of all
bounded linear operators between two separable Hilbert spaces $D$
and $E$. \ That is, we focus on the cases where the above-mentioned
Banach space $X$ is $\mathcal{B}(D,E)$.

A study on the boundary values of functions in Banach-space-valued Hardy spaces
$H^p(\mathbb D, X)$ of the unit disk was initiated in 1976 by A.V.
Bukhvalov \cite{Bu}. \ Since then, many researchers
have studied the spaces of boundary values of functions in $H^p(\mathbb D, X)$
(see the bibliographical references at the end of this paper). \ In
particular, in 1982 A.V. Bukhvalov and A.A. Danilevich
\cite{BD} showed that if a Banach space $X$ has the analytic
Radon-Nikod\'ym property (ARNP) (or equivalently, every function in
$H^1(\mathbb D,X)$ has radial limits a.e. on $\mathbb T$; cf.
\cite{BL}, \cite{BD}, \cite{DE}, \cite{Do2}, \cite{ED}), then the
space of boundary values of functions in $H^p(\mathbb D, X)$ is
$H^p(\mathbb T, X)$; more precisely, the mapping $h\mapsto bh$ is an
isometric isomorphism from $H^p(\mathbb D, X)$ onto $H^p(\mathbb T,
X)$ and moreover, $P[bh]=h$, where $P[\cdot]$ denotes the Poisson
integral, or equivalently, the mapping $f\mapsto P[f]$ is an
isometric isomorphism from $H^p(\mathbb T, X)$ {\it onto}
$H^p(\mathbb D, X)$. \ However, this is no longer true for spaces of
operator-valued functions. \ Indeed, if $X=\mathcal B(D,E)$, then
$X$ need not satisfy the ARNP in general, so that we cannot
guarantee that the mapping $f\mapsto P[f]$ is an isometric
isomorphism from $H^p(\mathbb T, X)$ onto $H^p(\mathbb D, X)$. \ In
fact, for each $1\leq p\leq\infty$, there exists a function $h\in
H^p(\mathbb D, \mathcal B(\ell^2))$ such that $h \neq P[f]$ for any
$f \in H^p(\mathbb T, \mathcal B(\ell^2))$ (see Example \ref{ex23}).
\ Thus, the following problem remained unsolved until now:
\begin{equation}\label{mainq}
\hbox{Find the circle companion of $H^p(\mathbb D, \mathcal B(D,E))$
for $1\le p\leq \infty$.}
\end{equation}
Although not necessarily explicitly stated as an open problem, the
problem (\ref{mainq}) appears in Nikolski's book \cite[p. 62, lines
14-15]{Ni2}, where it is mentioned implicitly. \ In this paper, we
solve problem (\ref{mainq}). \ Our solution aims to shed
additional insights into the study of boundary values, and how the
Poisson transform serves as a bridge between those boundary values
and the initial Hardy space function. \ Towards our solution, we
introduce a new space $L^p_{sot}(\mathbb T, \mathcal B(X, Y))$
($1\le p\le \infty$) defined by the space of all (equivalence
classes of)  SOT measurable functions $f:\mathbb T\to \mathcal B(X,
Y)$ such that $N(f)\in L^p(\mathbb T)$ (where
$N(f)(z):=||f(z)||_{\mathcal B(X, Y)}$); we identify $f$ and $g$
when $f(z)=g(z)$ for almost all $z\in\mathbb T$. \ In this case, let
$$
||f||_{L^p_{sot}(\mathbb T,\,  \mathcal
B(X,Y))}:=||N(f)||_{L^p(\mathbb T)}.
$$
Also for $1\le p\le\infty$, let $H^p_{sot}(\mathbb T, \mathcal B(X,
Y))$ be defined by the space of functions in $L^p_{sot}(\mathbb T,
\mathcal B(X, Y))$ such that $f(\cdot)x \in H^p(\mathbb T, Y)$ for
every $x \in X$. \ On the other hand, we define the ``strong Poisson
integral" $P_s[f]$ of $f$ in $H^p_{sot}(\mathbb T, \mathcal B(X,Y))$
by
$$
P_s[f](\zeta)x :=P[f(\cdot)x](\zeta)\quad(x\in X,\
\zeta\in\mathbb D).
$$
The aim of this paper is to prove that for $1\le p \le\infty$, the
mapping $f\mapsto P_s[f]$ is an isometric isomorphism from
$H^p_{sot}(\mathbb T, \mathcal B(D,E))$ onto $H^p(\mathbb D,
\mathcal B(D,E))$. \ In \cite[p. 53, Theorem 3.11.10]{Ni2} it is shown
that the mapping $f\mapsto P_s[f]$ provides an isometric isomorphism
from $H^\infty_{WOT}(\mathbb T, \mathcal B(D,E))$ onto
$H^\infty(\mathbb D, \mathcal B(D,E))$ when $D$ and $E$ are
separable Hilbert spaces - in fact, we can show that
$H^\infty_{WOT}(\mathbb T, \mathcal B(D,E))= H^\infty_{sot}(\mathbb
T, \mathcal B(D,E))$ in our language. \ This provides a sound
rationale for denoting this new space as $H^p_{sot}(\mathbb T,
\mathcal B(D,E))$, in a manner fully consistent with the well-known
result. \ In fact, we can get a more general version of the Banach
space setting. \ The following is the main result of this paper.

\smallskip

\begin{thm}\label{mainthm}
Let $X$ be a separable Banach space and $Y$ be a Banach space
satisfying the analytic Radon-Nikod\'ym property. \ Then, for $1\leq
p\leq \infty$, the mapping $f \mapsto P_s[f]$ is an isometric
isomorphism  from $H^p_{sot}(\mathbb T, \mathcal B(X, Y))$ onto
$H^p(\mathbb D, \mathcal B(X, Y))$.
\end{thm}

The following corollary is immediate from Theorem \ref{mainthm}.

\begin{cor}
Let $D$ and $E$ be separable Hilbert spaces. \ Then,
for $1\leq p\leq \infty$, the mapping $f \mapsto P_s[f]$ is an
isometric isomorphism from $H^p_{sot}(\mathbb T, \mathcal B(D, E))$
onto $H^p(\mathbb D, \mathcal B(D, E))$. \ As a result, $H^p_{sot}(\mathbb T, \mathcal B(D, E))$ is the circle companion of  $H^p(\mathbb D, \mathcal B(D, E))$. 
\end{cor}

\medskip
In Section 2, we give a few essential facts that will be needed to
prove Theorem \ref{mainthm}. \ Section 3 is devoted to a proof of
Theorem \ref{mainthm}. \ In the Appendix, we consider relevant
results for strong $H^p$-spaces.

\medskip

%
%
%
%

\section{Preliminaries}

We review here the preliminary background needed to prove the main
theorem, using \cite{TJML} and \cite{Ni2} as general references.
Let $m$ be the normalized Lebesgue measure on $\mathbb
T$. \ For a Banach space $X$, a function $f:\mathbb T \rightarrow X$
is said to be essentially separably valued if there exists a
Lebesgue measurable set $\mathbb T^\prime\subseteq \mathbb T$ such
that the range $f(\mathbb T^\prime)$ is separable and $m(\mathbb
T\setminus \mathbb T^\prime) =0$.

We begin with:

\medskip

\noindent {\bf Pettis Measurability Theorem} (\cite{TJML}). \ Let
$X$ be a  Banach space and $X^*$ denote the dual space of $X$. \ For
a function $f:\mathbb T \rightarrow X$, the following  are
equivalent:
\begin{itemize}
\item[(a)] $f$ is  strongly measurable
(i.e., there exists a sequence of simple functions $f_n$ such that
$f(z)=\lim_{n \to \infty} f_n(z)$ for almost all $z \in \mathbb T$);
\item[(b)] $f$ is essentially separably valued and
 weakly measurable (i.e., the
mapping $z\mapsto \langle f(z), x^*\rangle$ is Lebesgue measurable
for every $x^*\in X^*$).
\end{itemize}

\medskip

\noindent {\bf Observation}. \ By the Pettis Measurability Theorem,
the almost everywhere limit of a sequence of strongly measurable
functions is also strongly  measurable.

\medskip

Given a  function $f:\mathbb{T} \rightarrow X$, let
$$
N(f)(z):=\left\|f(z)\right\|_X.
$$
For $1\leq p\leq\infty$, let $L^p(\mathbb T,X)$ be the  space of all
(equivalence classes of) strongly measurable functions $f:\mathbb T
\rightarrow X$ such that $N(f)\in  L^p(\mathbb T)$. \ Endowed with
the norm
$$
||f||_{L^p(\mathbb T,\,  X)}:=||N(f)||_{L^p(\mathbb T)},
$$
the space $L^p(\mathbb T, X)$ is a Banach space. \ For $f \in
L^1(\mathbb T, X)$, the $n$-th Fourier coefficient of $f$, denoted
by $\widehat{f} (n)$, is defined by
$$
\widehat{f}(n):=\int_{\mathbb T}\overline{z}^{n}f(z)\, dm(z) \quad
\hbox{for each $n\in\mathbb Z$},
$$
where the integral is understood in the sense of the Bochner
integral. \ Also, $H^p(\mathbb T, X)$ is defined as the space of
functions $f\in L^p(\mathbb T, X)$ with $\widehat{f}(n)=0$ for
$n<0$.

\medskip

Hereafter, let $X$ and $Y$ be  Banach spaces and $\mathcal B(X, Y)$
denote the space of all bounded linear operators from $X$ to $Y$,
and abbreviate $\mathcal B(X, X)$ as $\mathcal B(X)$ \ We write
$\hbox{Hol}(\mathbb D, X)$ for the set of all $X$-valued functions
holomorphic in $\mathbb D$.

\medskip

\noindent{\bf Equivalent conditions of holomorphic functions}
(\cite{Ni2}). \ If $h:\mathbb D \rightarrow \mathcal B(X, Y)$, then
the following are equivalent.
\begin{itemize}
\item[(a)] $h \in \hbox{Hol}(\mathbb D, \mathcal B(X, Y))$;
\item[(b)] $h(\cdot)x\in \hbox{Hol}(\mathbb D, Y)$  for all $x \in X$;
\item[(c)] $\langle h(\cdot)x, y^* \rangle \in \hbox{Hol}(\mathbb D, \mathbb C)$  for all $x \in
X$ and $y^* \in Y^*$.
\end{itemize}

\medskip

Let us associate to any function $h: \mathbb D \to X$,
a family of functions $h_r$ on $\mathbb T$, defined by
$$
h_r(z):=h(rz)  \quad (z \in \mathbb{T}, \; 0\leq  r <1).
$$
For $1\leq p \leq \infty$, let $H^p(\mathbb D, X)$ be the space  of
all functions $h \in \hbox{Hol}(\mathbb D, X)$ satisfying
$$
||h||_{H^p(\mathbb D,\, X)}:=\sup\bigl\{||N(h_r)||_{L^p(\mathbb T)}:
r<1 \bigr\}<\infty.
$$
Then $H^p(\mathbb D, X)$ is a Banach space (cf. \cite{Do2}). \ If
$h \in \hbox{Hol}\,(\mathbb D, X)$, then we may write
$$
h(\zeta)=\sum_{n=0}^{\infty}x_n \zeta^n \quad(\zeta \in \mathbb D, \
x_n \in X).
$$
Hence for each $0\leq r<1$,
$$
h_r(z)=\sum_{n=0}^{\infty}x_n r^n z^n \quad(z \in \mathbb T),
$$
which implies that $h_r$ is essentially separably
valued. \ For each $x^* \in X^*$,
$$
\langle h_r(z), x^* \rangle=\sum_{n=0}^{\infty}\bigl \langle
x_nr^n,\, x^* \bigr\rangle z^n \quad(z \in \mathbb T),
$$
which implies that $h_r$ is weakly measurable. \  Thus, by the
Pettis Measurability Theorem, $h_r$ is strongly measurable.
Therefore we have that
$$
||h||_{H^p(\mathbb D,\, X)}=\sup_{0\leq r<1}||h_r||_{L^p(\mathbb
T,\,X)}.
$$

\medskip

For $f \in L^1(\mathbb T, X)$, let $P[f]$ denote the Poisson
integral of $f$ defined by
\begin{equation}\label{sdfhjudx}
P[f](\zeta):=\int_{\mathbb T}P_{\zeta}(z)f(z)dm(z) \quad (\zeta \in
\mathbb D),
\end{equation}
where $P_{\zeta}(z)$ is the Poisson kernel.

\medskip

The following are basic properties of Poisson integrals.

\begin{lem}\label{sxacdfghnm}\cite[Lemma 3.11.6.]{Ni2} If $f \in L^p(\mathbb T, X)$
($1\leq p\leq \infty$), then
\begin{itemize}
\item[(a)] $||(P[f])_r||_{L^p(\mathbb T,\, X)}\leq ||f||_{L^p(\mathbb
T,\, X)}$  \ for all $0\leq r<1$;
\item[(b)] If $p < \infty$, then $\lim_{r \to
1}||(P[f])_r-f||_{L^p(\mathbb T,\, X)}=0$;
\item[(c)] $\lim_{r \to 1}||(P[f])_r(z)-f(z)||_X=0$  \ for almost all $z \in \mathbb
T$.
\end{itemize}
\end{lem}

\medskip

On the other hand, the function $P: H^p(\mathbb T, X) \to
H^p(\mathbb D, X)$ given by (\ref{sdfhjudx}), is an isometry for all
$1 \le p \le \infty$ (cf. \cite{BL}). \ As we noticed in
the introduction, if $X$ has the ARNP and $1\leq p\leq\infty$, then
the function $P: H^p(\mathbb T, X) \to H^p(\mathbb D, X)$ given by
(\ref{sdfhjudx}) is an isometric isomorphism (cf. \cite{BD}). \
However, the function $P:H^p(\mathbb T, \mathcal B(D,E))\to
H^p(\mathbb D, \mathcal B(D,E))$ given by (\ref{sdfhjudx}) is not
onto in general, as we see in the following example.

\medskip

\begin{ex}\label{ex23} Let $h:\mathbb D \rightarrow
\mathcal B(\ell^2)$ be defined by $(h(\zeta)x)(n):=\zeta^nx(n)$ for
each $x\in \ell^2$. \ Then $h \in \hbox{Hol}(\mathbb D, \mathcal
B(\ell^2))$ and $||h||_{H^{\infty}(\mathbb D,\, \mathcal B
(\ell^2))}= 1$, so that $h\in  H^p(\mathbb D, \mathcal B (\ell^2))$
for all $1\leq p\leq \infty$. \ Suppose that there
exists $p\in [1, \infty]$ such that $P:H^p(\mathbb T, \mathcal
B(\ell^2))\to H^p(\mathbb D, \mathcal B(\ell^2))$ is onto. \ Then
there exists a function $f \in H^p(\mathbb T, \mathcal B(\ell^2))$
such that $P[f]=h$. \ For each $z \in \mathbb T$, define a ``strong
boundary function" $bh:\mathbb T\to \mathcal B(\ell^2)$ by
\begin{equation}\label{vedkwk}
(bh)(z)x:=\lim_{r \to 1}h_r(z)x=(z^n x(n)) \quad (x\equiv (x(n)) \in
\ell^2).
\end{equation}
Then it follows from Lemma \ref{sxacdfghnm}(c) that for all $x \in
\ell^2$,
$$
(bh)(z)x=\lim_{r \to 1}(P[f])_r(z)x=f(z)x
$$
for almost all $z \in \mathbb T$, which implies $f=bh$. \ Let $z_1
\neq z_2$ in $\mathbb T$. \ For $k=1,2$, write $z_k=e^{i\theta_k}$
($0\leq \theta_k< 2\pi$). \ Then there exists $n_0\in\mathbb N$ such
that $\frac{\pi}{2}<n_0|\theta_2-\theta_1|\leq\pi$ ($\hbox{mod} \
2\pi$). \ Let $\{e_n:n=1,2,\cdots\}$ be the canonical orthonormal
basis for $\ell^2$. \ Then it follows from (\ref{vedkwk}) that
$$
\left\|\bigl(f(z_1)-f(z_2)\bigr)
e_{n_0}\right\|_{\ell^2}=|z_1^{n_0}-z_2^{n_0}|
=|1-(z_2\overline{z_1})^{n_0}|
>\sqrt{2},
$$
which implies that $f$ is not essentially separably valued. \ Thus,
by the Pettis Measurability Theorem, $f$ is not strongly measurable,
a contradiction. \ Therefore, $P:H^p(\mathbb T, \mathcal
B(\ell^2))\to H^p(\mathbb D, \mathcal B(\ell^2))$ is not onto for
any $p\in[1,\infty]$. \qed
\end{ex}

\bigskip

%
%
%
%

\section{Proof of the main result}

\medskip

A function $f:\mathbb T\to \mathcal B(X,Y)$ is called {\it SOT
measurable} if the mapping $z\mapsto f(z)x$ is strongly measurable
for every $x\in X$.

We introduce a new normed space.

\begin{df} For $1\le p\le\infty$, define $L^p_{sot}(\mathbb T, \mathcal B(X,
Y))$ by the space of all (equivalence classes of)  SOT measurable
functions $f:\mathbb T\to \mathcal B(X, Y)$ such that $N(f)\in
L^p(\mathbb T)$; we identify $f$ and $g$ when $f(z)=g(z)$ for almost
all $z\in\mathbb T$. \ In this case, define
$$
||f||_{L^p_{sot}(\mathbb T,\,  \mathcal
B(X,Y))}:=||N(f)||_{L^p(\mathbb T)}.
$$
\end{df}

\medskip

We can easily check that $L^p_{sot}(\mathbb T, \mathcal B(X, Y))$ is
a normed space and \linebreak $L^q_{sot}(\mathbb T, \mathcal B(X,
Y))\subseteq L^p_{sot}(\mathbb T, \mathcal B(X, Y))$ if $1\le p\le
q\le\infty$. \ Further, the space $L^p_{sot}(\mathbb T, \mathcal
B(X,Y))$ is  a Banach space.

\medskip

\begin{lem}\label{lem377}
For $1\le p\le\infty$, $L^p_{sot}(\mathbb T,\, \mathcal B(X, Y))$ is
a Banach space.
\end{lem}
\begin{proof} The proof follows from a slight variation of the standard proof (cf.
\cite{Ru}) for the completeness of scalar-valued $L^p$-spaces,
except for SOT-measurability. \ To be completely
rigorous, we sketch a proof of the validity of SOT-measurability.

Suppose $(f_n)$ is a Cauchy sequence in $L^p_{sot}(\mathbb T,\,
\mathcal B(X, Y))$. \ Then we can choose a subsequence $(f_{n_i})$
such that
$$
||f_{n_{i+1}}-f_{n_i}||_{L^p_{sot}(\mathbb T,\, \mathcal B(X,
Y))}<2^{-i} \ \hbox{for all} \ i=1,2,3,\cdots.
$$
If we put $g:=\sum_{i=1}^\infty (f_{n_{i+1}}-f_{n_i})$, then it is
easy to show that $g(z)\in \mathcal B(X,Y)$ for almost all
$z\in\mathbb T$ and in turn,
$$
f(z):=f_{n_1}(z)+g(z)
$$
converges for almost all $z \in \mathbb T$. \ Therefore for each
$x\in X$, $f(z)x=\lim_{i\to\infty} f_{n_i}(z)x$ for almost all
$z\in\mathbb T$. \ Since $f_{n_i}$ is SOT measurable, the mapping
$z\mapsto f_{n_i}(z)x$ is strongly measurable, so that the mapping
$z\mapsto f(z)x$ is also strongly measurable. \ Therefore $f$ is SOT
measurable.
\end{proof}

\medskip

\begin{rem}\label{35111}
In the definition of $L^p_{sot}(\mathbb T, \mathcal B(X, Y))$, we
implicitly suppose $N(f)$ is (Lebesgue) measurable. \ In fact, we
don't guarantee that if $f$ is SOT measurable then $N(f)$ is
measurable in general. \ To see this, let $\ell^2(\mathbb T)$ be the
set of all functions $x: \mathbb T \to \mathbb C$ such that $x(z)=0$
for all but a countable number of $z$'s and $\sum_{z \in \mathbb
T}|x(z)|^2 <\infty$. \ For $x$ and $y$ in $\ell^2(\mathbb T)$ define
$$
\langle x, y \rangle:=\sum_{z\in \mathbb T}x(z)\overline{y(z)}.
$$
 Then
$\ell^{2}(\mathbb T)$ is a (non-separable) Hilbert space. \ Let $F$
be a nonmeasurable  set in $\mathbb T$.\ For $z \in \mathbb T$, let
$f:\mathbb T\to \mathcal B(\ell^2(\mathbb T))$ be defined by
$$
(f(z)x)(s):=\begin{cases} x(z), \quad \  \hbox{if} \ s=z \in F\\
0,  \qquad  \  \ \hbox{if} \  z \notin F \ \hbox{or} \  s\neq z.
\end{cases}
$$
Then for each $x \in \ell^2(\mathbb T)$,  we have that $f(z)x=0$ for
almost all $z \in \mathbb T$, and hence $f$ is SOT measurable.

We now claim that
\begin{equation}\label{36003}
N(f)=\mathbf{1}_F  \quad ({\mathbf 1}_F \ \hbox{denotes the
indicator function of the set} \  F),
\end{equation}
which implies that $N(f)$ is not measurable because $F$ is a non
measurable set. \ To see this, for each $z \in \mathbb T$, let
$$
x_z(s):=\begin{cases} 1,  &\hbox{if} \ s=z\\
0,  &\hbox{if} \ s \neq z.
\end{cases}
$$
Then, $x_z \in \ell^2(\mathbb T)$ and $||x_z||=1$. \ If $z \in F$,
then $ (f(z)x_z)(s)=x_z(s)$, so that $||f(z)x_z||_{\ell^2(\mathbb
T)}=1$. \ But since $f(z)$ is a contraction, it follows that
$N(f)(z)=1$ for all $z \in F$.
If instead $z \notin F$, then $f(z)=0$, so that $N(f)(z)=0$. \ This
proves (\ref{36003}). \qed
\end{rem}

\medskip
 We note that in the above remark, $\ell^2(\mathbb T)$ is not
separable. \ However, we can show  that the
SOT-measurability of $f$ implies the measurability of $N(f)$ if $X$
is a separable Banach space.

\begin{lem}\label{sdsacdny}
Let $X$ be a separable Banach  space. \ If $f:\mathbb T \to \mathcal
B(X, Y)$ is SOT measurable then $N(f)$ is measurable.
\end{lem}

\begin{proof} Suppose that  $f:\mathbb T \to \mathcal B(X, Y)$ is
SOT measurable. \ Then for all $x \in X$, the mapping $z
\mapsto f(z)x$ is strongly measurable, and hence the
mapping $z \mapsto ||f(z)x||$ is measurable. \ Thus the
mapping $z \mapsto \frac{||f(z)x||}{||x||}$ is
measurable for all nonzero $x \in X$. \ Choose a
countable dense subset $X_0$ of $X$. \ Then we can easily see that
$$
N(f)(z)=\sup \biggl\{\frac{||f(z)x||}{||x||}: 0 \neq
 x\in X_0\biggr\}.
$$
Thus the mapping $z\mapsto N(f)(z)$ is measurable.
\end{proof}

\medskip

We now introduce a space which fits our purpose:

\begin{df} For $1\le p\le \infty$, let
$H^p_{sot}(\mathbb T, \mathcal B(X, Y))$ be the space of  all
(equivalence classes of) functions $f\in L^p_{sot}(\mathbb T,
\mathcal B(X, Y))$ such that $f(\cdot)x\in H^p(\mathbb T,Y)$ for
every $x\in X$.
\end{df}

\medskip

Observe that for $1\le p\le\infty$, $H^p_{sot}(\mathbb T,\, \mathcal
B(X, Y))$ is a closed subspace of  $L^p_{sot}(\mathbb T,\, \mathcal
B(X, Y))$, so that by Lemma \ref{lem377}, $H^p_{sot}(\mathbb T,\,
\mathcal B(X, Y))$ is a Banach space.

\begin{ex}\label{hsot}
In general, $H^p(\mathbb T,\mathcal B(X,Y))\ne H^p_{sot}(\mathbb T,
\mathcal B(X,Y))$ for all $1\leq p\leq \infty$. \ To see this, let
$H^2\equiv H^2(\mathbb T)$ and define the function $f:\mathbb T
\rightarrow \mathcal B(H^2)$ by
$$
f(z)x(s):=x(zs).
$$
Since the set of all polynomials on $\mathbb T$ is dense in
$H^2$, it follows that the mapping
$z \mapsto f(z)x$ is (uniformly) continuous for each $x \in H^2$. \
Thus, by the Pettis Measurability Theorem,  $f$ is SOT measurable. \
Since $ N(f)(z)=1$ for all $z \in \mathbb T$, it follows that $f \in
L^{\infty}_{sot}(\mathbb T, \mathcal B(H^2))$ with
$||f||_{L^{\infty}_{sot}(\mathbb T,\, \mathcal B(H^2))}=1$. \
Moreover for each $x\in H^2$ and $n\in \mathbb Z$,
$$
\bigl(\widehat{f}(n)x \bigr)(s)=\int_{\mathbb
T}\overline{z}^nf(z)x(s)dm(z)=\bigl \langle x(zs),  \ z^n \bigr
\rangle_{H^2} =\widehat{x}(n)s^n,
$$
which implies that $f \in H^{\infty}_{sot}(\mathbb T,
\mathcal B(H^2))\subseteq H^{p}_{sot}(\mathbb T, \mathcal B(H^2))$
for all $1\leq p\leq \infty$. \ However we have that
$f\notin H^p(\mathbb T, \mathcal B(H^2))$. \ To see
this we use the same argument as Example \ref{ex23}. \ Let $z_1 \neq
z_2$ in $\mathbb T$. \ Write $z_k=e^{i\theta_k}$ ($0\leq \theta_k<
2\pi$). \ Then there exists $n_0\in\mathbb N$ such that
$\frac{\pi}{2}<n_0|\theta_2-\theta_1|\leq\pi$ ($\hbox{mod} \ 2\pi$).
We thus have
$$
\aligned ||(f(z_1)-f(z_2))s^{n_0}||_{H^2}^2&=\int_{\mathbb
T}|(z_1s)^{n_0}
-(z_2s)^{n_0}|^2dm(s)\\
&=\int_{\mathbb T}|1-(z_2\overline{z_1})^{n_0}|^2dm(s)>2,
\endaligned
$$
which implies that $f$ is not essentially separably valued. \ Thus,
by the Pettis Measurability Theorem, $f$ is not strongly measurable,
so that, $f \notin H^p(\mathbb T, \mathcal B(H^2))$.
\qed
\end{ex}

\medskip

\begin{df}\label{dfecfregvryhb}
For $f \in H^1_{sot}(\mathbb T, \mathcal B(X, Y))$ and $x \in X$,
let $P_s[f](\cdot)x:\mathbb D \rightarrow Y$ be defined by
$$
P_s[f](\zeta)x:=P[f(\cdot)x](\zeta) \quad (\zeta \in \mathbb D),
$$
where $P[\cdot]$ denotes the Poisson integral. \ In this
case, $P_s[f]$ is called the {\it strong Poisson integral} of $f$.
\end{df}

\medskip

\begin{lem}\label{scxwesrftvsryhik}
For $1\leq p\leq \infty$, the mapping $f \mapsto P_s[f]$ is a
contraction from $H^p_{sot}(\mathbb T, \mathcal B(X, Y))$ to
$H^p(\mathbb D, \mathcal B(X, Y))$.
\end{lem}

\begin{proof}
Let $f \in H^p_{sot}(\mathbb T, \mathcal B(X, Y))$ ($1\leq p\leq
\infty$) and $\zeta=re^{i \theta} \in \mathbb D$. \ Clearly,
$P_s[f](\zeta)$ is linear on $X$. \ For each $x \in X$,
$$
\aligned ||P_s[f](\zeta)x||&=\Bigl|\Bigl|\int_{\mathbb T}P_{\zeta}(z)f(z)xdm(z)\Bigr|\Bigr|\\
&\leq \frac{1+r}{1-r} \cdot||f||_{L^1_{sot}(\mathbb T,\, \mathcal
B(X, Y))}||x||,
\endaligned
$$
which implies that $P_s[f](\zeta) \in \mathcal B(X, Y)$.  \ Since
$P_s[f](\cdot)x\in H^1(\mathbb D, Y)$ for every $x\in X$, it follows
$P_s[f] \in \hbox{Hol}\, (\mathbb D,  \mathcal B(X, Y))$. \ We now
claim that
$$
P_s[f] \in  H^p(\mathbb D,\, \mathcal B(X, Y)) \ \hbox{and} \
||P_s[f]||_{H^p(\mathbb D,\, \mathcal B(X, Y))}\leq
||f||_{L^p_{sot}(\mathbb T,\, \mathcal B(X, Y))}.
$$
For each $\zeta \in
\mathbb D$ and a unit vector $x\in X$,
$$
||P_s[f](\zeta)x|| \leq \int_{\mathbb T}P_{\zeta}(z)||f(z)||dm(z)
=P[N(f)](\zeta).
$$
Thus $||P_s[f](\zeta)||\leq P[N(f)](\zeta)$ for all $\zeta \in
\mathbb D$ and hence, by Lemma \ref{sxacdfghnm}(a), we have
$$
||(P_s[f])_r||_{L^p(\mathbb T,\, \mathcal B(X, Y))}\leq \bigl|
\bigl|(P[N(f)])_r\bigr|\bigr|_{L^p(\mathbb T)}\le
||f||_{L^p_{sot}(\mathbb T,\, \mathcal B(X, Y))},
$$
which implies that $P_s[f]\in H^p(\mathbb D, \mathcal B(X, Y))$ and
$$
||P_s[f]||_{H^p(\mathbb D,\, \mathcal B(X, Y))}\leq
||f||_{L^p_{sot}(\mathbb T,\, \mathcal B(X, Y))}.
$$
This completes the proof.
\end{proof}

\medskip

We are ready to prove our main theorem. \ Before doing it, we would
like to underline a reason why our proof is little intricate. \ Let $h
\in H^1(\mathbb D, \mathcal B(X, Y))$ and assume that $Y$ has the
ARNP. \  Since $h(\cdot)x \in H^1(\mathbb D, Y)$ for each $x\in X$,
there exists the following radial strong limit $bh$
a.e. on $\mathbb T$: i.e., for each $x\in X$,
$$
bh(z)x:=\lim_{r \to 1}h_r(z)x \quad (z \in \mathbb T).
$$
Write
$$
E_x:=\{z \in \mathbb T: bh(z)x \ \hbox{does not exist}\} \quad
\hbox{and} \quad E:=\bigcup_{x \in X}E_x.
$$
Then $m(E_x)=0$ for each $x \in X$, but we don't guarantee $m(E)=0$.
Thus the function $bh$ may not be defined almost everywhere on
$\mathbb T$. \ Therefore $bh$ is not appropriate for a boundary
function of $h$. \ The crucial point of our proof is how to construct
a ``boundary function" defined almost everywhere on $\mathbb T$ for
a function in $H^p(\mathbb D, \mathcal B(X, Y))$.

\medskip

We will now prove Theorem \ref{mainthm}, which we restate for the
reader's convenience:

\medskip

\noindent{\bf Theorem 1.1.}\label{sdfvghfnj} Let $X$ be a separable
Banach space and $Y$ be a Banach space satisfying the analytic
Radon-Nikod\'ym property. \ Then, for $1\leq p\leq \infty$, the
mapping $f \mapsto P_s[f]$ is an isometric isomorphism  from
$H^p_{sot}(\mathbb T, \mathcal B(X, Y))$ onto $H^p(\mathbb D,
\mathcal B(X, Y))$.

\medskip

\begin{proof}  Let $X$ be a separable
Banach space and $Y$ be a Banach space satisfying the analytic
Radon-Nikod\'ym property. \ Let $h \in H^1(\mathbb D, \mathcal B(X,
Y))$.  \ Our first task is to define a ``boundary function" $b_sh$
a.e. on $\mathbb T$ for $h$. \ To do so, let $X_0=\{x_n \in
X:n=1,2,\cdots\}$ be a countable dense subset of $X$. \ Then for
each $n=1,2,\cdots$, there exists a measurable set
$E_n$ with $m(E_n)=0$ such that $bh(z)x_n=\lim_{r \to
1}h_r(z) x_n$ exists for all $z \in \mathbb T \setminus E_n$. \
Then $bh(\cdot)x_n\in H^1(\mathbb T, Y)$ for each
$n=1,2,\cdots$. \ Put $E_0:=\cup_{n\geq 1}E_n$. \ Then
$m(E_0)=0$. \ For $z \in \mathbb T \setminus E_0$, let
\begin{equation}\label{edacvvgtehtry}
q(z):=\sup\left\{\frac{||bh(z)x||}{||x||}: 0 \neq x\in X_0\right\}.
\end{equation}
Observe that for all $z \in \mathbb T \setminus E_0$ and each $x \in
X_0$,
\begin{equation}\label{xsdfdsadfdcsxazscd}
||bh(z)x||=\lim_{r \to 1}||h_r(z)x||\leq \liminf_{r \to
1}||h_r(z)||\cdot||x||.
\end{equation}
Let $u(z):=\liminf_{r \to 1}N(h_r)(z)$.\ Since $h \in
H^1(\mathbb D, \mathcal B(X, Y))$, $N(h_r)$ is in
$L^1(\mathbb T)$ for each $0\leq r <1$, so that $u$ is measurable.
Also by (\ref{edacvvgtehtry}) and
(\ref{xsdfdsadfdcsxazscd}), we have
\begin{equation}\label{xscvzxfdcsxazscd}
0\leq q(z)\leq u(z) \ \hbox{for all} \  z \in \mathbb T \setminus
E_0.
\end{equation}
On the other hand, by Fatou's lemma, we have
$$
\int_{\mathbb T}u(z)dm(z)\leq \hbox{lim inf}_{r \to 1} \int_{\mathbb
T}N(h_r)(z)dm(z) \leq ||h||_{H^1(\mathbb D,\, \mathcal B(X,
Y))}<\infty,
$$
which implies that $u\in L^1(\mathbb T)$. \ Thus there exists a
subset  $E_u$ of $\mathbb T$ with $m(E_u)=0$ such that $u(z)<\infty$
for all $z \in \mathbb T \setminus E_u$. \ Hence, by
(\ref{xscvzxfdcsxazscd}), $q(z)\leq u(z)<\infty$ for  all $z \in
\mathbb T \setminus(E_0 \cup E_u)$. \ Therefore $bh(z)$ can be
extended to a bounded linear operator $b_sh(z)$ on $X$ for almost
all $z \in \mathbb T$: for each $z \in \mathbb T \setminus(E_0 \cup
E_u)$ and $x \in X$, define
\begin{equation}\label{dcrvg}
b_sh(z)x:=\lim_{n \to \infty}bh(z)x_n,
\end{equation}
where $(x_n)$ is a sequence in $X_0$ such that $x_n \to x$. \ We note
that (\ref{dcrvg}) is independent of the particular choice of the
dense subset $X_0$ of $X$ and a sequence $(x_n)$ in $X_0$ :
indeed let $Y_0$ is another countable dense subset of
$X$ and $(y_n)$ is a sequence in $Y_0$ such that $y_n \to x$. \ By the
same argument above, we see that for almost all $z \in \mathbb T$,
$$
q^{\prime}(z):=\sup\left\{\frac{||bh(z)x||}{||x||}: 0 \neq x\in X_0
\cup Y_0\right\}<\infty.
$$
Thus
$$
||bh(z)x_n- bh(z)y_n|| \leq q^{\prime}(z)||x_n- y_n||
\to 0 \quad \hbox{as} \ n \to \infty,
$$
which implies that the function $b_sh(z)$ is well-defined on $X$ for
almost all $z \in \mathbb T$. \ (We call $b_sh$ the {\it strong
boundary function} of $h$.)

Now let $1\le p\le \infty$ and suppose  $h \in H^p(\mathbb D,
\mathcal B(X, Y))$. \ Then $b_sh(z) \in \mathcal B(X,Y)$
for almost all $z \in \mathbb T$ and it is easy to show that $b_sh$
is SOT measurable and hence, by Lemma \ref{sdsacdny}, $N(b_sh)$ is
measurable because $X$ is separable. \ We claim that
\begin{equation}\label{trytry}
b_sh\in H^p_{sot}(\mathbb T, \mathcal B(X, Y)).
\end{equation}
To see this, we first observe that, by
(\ref{xscvzxfdcsxazscd}), $N(b_sh)(z)=q(z)\leq \liminf_{r \to
1}N(h_r)(z)$ for almost all $z \in \mathbb T$. \ Thus for $1\leq p
<\infty$, it follows from Fatou's lemma that
\begin{equation}\label{sxcfvllllll}
\aligned \int_{\mathbb T}N(b_sh)(z)^pdm(z)&\leq \hbox{lim inf}_{r
\to
1} \int_{\mathbb T}N(h_r)(z)^pdm(z)\\
& \leq ||h||_{H^p(\mathbb D,\, \mathcal B(X, Y))}^p<\infty,
\endaligned
\end{equation}
Let $x \in X$ be arbitrary and $(x_n)$ be a sequence in $X_0$ such
that $x_n \to x$. \ Then it follows from (\ref{sxcfvllllll}) that
$$
\aligned ||b_sh(\cdot) x-bh(\cdot) x_n||_{L^p(\mathbb T, Y)}&=
\biggl(\int_{\mathbb T}||b_sh(z)(x-x_n)||^pdm(z)\biggr)^{\frac{1}{p}}\\
&\leq||h||_{H^p(\mathbb D,\, \mathcal B(X, Y))}||x-x_n|| \to 0 \quad
\hbox{as} \  n \to \infty.
\endaligned
$$
But since $H^p(\mathbb T,Y)$ is a closed subspace of $L^p(\mathbb
T,Y)$ and $bh(\cdot)x_n\in H^p(\mathbb T,Y)$, we have
$b_sh(\cdot)x\in H^p(\mathbb T,Y)$, which together with
(\ref{sxcfvllllll}) implies that $b_sh \in H^p_{sot}(\mathbb T,
\mathcal B(X, Y))$ and
$$
||b_sh||_{H^p_{sot}(\mathbb T, \mathcal B(X, Y))}\leq
||h||_{H^p(\mathbb D,\, \mathcal B(X, Y))}.
$$
If instead $p=\infty$, then  $h \in H^1(\mathbb D, \mathcal B(X,
Y))$, so that $b_sh \in H^1_{sot}(\mathbb T, \mathcal B(X, Y))$. \
Also, it follows from (\ref{xsdfdsadfdcsxazscd}) that
\begin{equation}\label{adcvfbghbkkgjjdchj}
||b_sh||_{L^{\infty}_{sot}(\mathbb T,\, \mathcal B(X, Y))}\leq
||h||_{H^{\infty}(\mathbb D,\, \mathcal B(X, Y))}.
\end{equation}
Thus $b_sh \in H^{\infty}_{sot}(\mathbb T, \mathcal B(X, Y))$. \
This proves (\ref{trytry}).

We next claim that
\begin{equation}\label{314003}
P_s[b_sh]=h.
\end{equation}
Let $x \in X$ be arbitrary. \ Then for each $\zeta=re^{i\theta} \in
\mathbb D$,
\begin{equation}\label{sdfghynhbgfvdcxs}
\aligned ||P_s[b_sh](\zeta)x||&\leq \frac{1+r}{1-r}\int_{\mathbb
T}||b_sh(z)x||dm(z)\\
&= \frac{1+r}{1-r}\cdot ||b_sh||_{ L^1_{sot}(\mathbb T, \ \mathcal
B(X, Y))} \cdot ||x||,
\endaligned
\end{equation}
 Choose a sequence $(x_n)$ in $X_0$ such
that $x_n \to x$. \ Then for each $\zeta \in \mathbb D$,
$$
h(\zeta)x=\lim_{n\to\infty} h(\zeta)x_n =\lim_{n\to\infty}
P_s[b_sh](\zeta)x_n =P_s[b_sh](\zeta)x,
$$
where the last equality follows from (\ref{sdfghynhbgfvdcxs}). \
This proves (\ref{314003}). \ Thus the mapping $f \mapsto P_s[f]$ is
a surjection from $H^p_{sot}(\mathbb T, \mathcal B(X, Y))$ to
$H^p(\mathbb D, \mathcal B(X, Y))$. \ Therefore, by Lemma
\ref{scxwesrftvsryhik}, (\ref{sxcfvllllll}) and
(\ref{adcvfbghbkkgjjdchj}), the mapping $f \mapsto P_s[f]$ is an
isometry from $H^p_{sot}(\mathbb T, \mathcal B(X, Y))$ onto
$H^p(\mathbb D, \mathcal B(X, Y))$. \ This completes the proof.
\end{proof}

\medskip

Theorem \ref{mainthm} may  fail if
 the separability condition on $X$
is dropped. \ For $z \in \mathbb T$ and $x \in \ell^2(\mathbb T)$,
let $f:\mathbb T\to \mathcal B(\ell^2(\mathbb T))$ be defined by
$$
(f(z)x)(s):=\begin{cases} x(z), \quad \  \hbox{if} \ s=z \\
0,  \qquad  \  \ \hbox{if} \  s\neq z.
\end{cases}
$$
Then by the argument in Remark \ref{35111}, we have $N(f)=1$, and
hence $f \in H^{p}_{sot}(\mathbb T, \mathcal B(\ell^2(\mathbb T)))$
with $||f||_{H^{p}_{sot}(\mathbb T, \mathcal B(\ell^2(\mathbb
T)))}=1$ for all $1 \leq p \leq \infty$. \ Since $(f(z)x)(s)$ is
zero for all $z \ne s$, it follows that for each $x \in
\ell^2(\mathbb T)$, $\zeta \in \mathbb D$ and $s \in \mathbb T$,
$$
\aligned \bigl(P_s[f](\zeta)x\bigr)(s)
=\int_{\mathbb T}P_{\zeta}(z)(f(z)x)(s)dm(z) =0,
\endaligned
$$
which implies that $P_s[f]=0$ in $H^{p}(\mathbb D, \mathcal
B(\ell^2(\mathbb T)))$. \ Therefore, the mapping $f \mapsto P_s[f]$ is
{\it not} an isometry.

\bigskip

We conclude with consideration on adjoints of functions in
$H^p_{sot}(\mathbb T, \mathcal B(X,Y))$.

\medskip

For a function $f:\mathbb T\to \mathcal B(X,Y)$, define
the ``adjoint" $f^*:\mathbb T\to \mathcal B(Y^*,X^*)$ of $f$ by
$$
f^*(z):=f(z)^*\quad (z\in\mathbb T).
$$
We may ask the following question: for $1\le p\le
\infty$, does it follow that
$$
f\in H^p_{sot}(\mathbb T, \mathcal B(X,Y))\ \Longrightarrow\ f^*\in
H^p_{sot}(\mathbb T, \mathcal B(Y^*,X^*))\,?
$$
In the sequel, we give an affirmative answer to this question if $X$
is  reflexive. \ To begin with we review some definitions.

A function $f:\mathbb T \rightarrow X$ is called {\it weakly
integrable} if $\langle f,  x^*\rangle \in L^1(\mathbb T)$ for every
$x^* \in X^*$. \ If $f$ is weakly  integrable then the function
$T_f:X^* \to L^1(\mathbb T)$, defined by $T_f x^*:=\langle f,
x^*\rangle$, is a bounded linear operator. \ A weakly  integrable
function $f:\mathbb T\to X$ is called {\it Pettis integrable} if the
adjoint $T_f^*$ of the operator $T_f$ maps $L^{\infty}(\mathbb T)$
into $X$. \ It is well known that
$$
f\ \hbox{is Bochner integrable} \Longrightarrow f\ \hbox{is Pettis
integrable} \Longrightarrow f\ \hbox{is weakly integrable.}
$$
Also it is known (cf. \cite[Proposition 1.2.36.]{TJML}) that for a
weakly integrable function $f: \mathbb T \rightarrow X$,
the following are equivalent:
\begin{itemize}
\item[(a)] $f$ is Pettis  integrable;
\item[(b)] for each
measurable set $B$ in $\mathbb T$, there exists an element $x_B \in
X$ such that for every $x^*\in X^*$ we have $ \langle x_B,
x^*\rangle=\int_{B} \langle f(z), x^* \rangle dm(z)$.
\end{itemize}
In this case, we shall write
$$
x_B =: (p)\hskip -.1cm-\hskip -.2cm\int_B f(z)dm(z),
$$
and call it the {\it Pettis integral} of $f$ over $B$.

We then have:

\begin{lem}\label{gbfyjnlhdsas}  Let $X$ be a reflexive Banach space, and let $f
\in L^1_{sot}(\mathbb T, \mathcal B(X, Y))$. \ Then for each $y^* \in Y^*$
and $\zeta \in \mathbb D$, we have
$$
P_s[f]^*(\zeta)y^*=(p)\hskip -.1cm-\hskip -.2cm \int_{\mathbb
T}P_{\zeta}(z)f^*(z)y^*dm(z),
$$
where $P_s[f]^*(\zeta):=P_s[f](\zeta)^*$.
\end{lem}

\begin{proof} Let $X$ be reflexive and $f
\in L^1_{sot}(\mathbb T, \mathcal B(X, Y))$. \ Then for all $x \in
X$ and $y^* \in Y^*$, the mapping $z \mapsto
\bigl\langle x, f^*(z)y^* \bigr \rangle=\bigl\langle f(z)x, y^*
\bigr \rangle $ is measurable. \ But since $X$ is reflexive,
the mapping $z \mapsto f^*(z)y^*$ is weakly measurable.
\ Thus the mapping $z \mapsto P_{\zeta}(z)f^*(z)y^*$ is
weakly measurable for each $\zeta=re^{i\theta} \in \mathbb D$. \ For
each $x \in X$,
$$
\aligned \int_{\mathbb T}\bigl|\langle x,
P_{\zeta}(z)f^*(z)y^*\rangle \bigr|dm(z)&=\int_{\mathbb
T}P_{\zeta}(z)\bigl| \langle f(z)x, y^* \bigr\rangle \bigr|
dm(z)\\
&\leq \frac{1+r}{1-r} \cdot ||y^*||\cdot||x||
\cdot||f||_{L^1_{sot}(\mathbb T,\, \mathcal B(X, Y))}\\
&<\infty,
\endaligned
$$
which implies that $P_{\zeta}(\cdot)f^*(\cdot)y^*$ is weakly
integrable and hence Pettis integrable. \ Thus for all $x \in X$ and
$\zeta \in \mathbb D$,
$$ \aligned \bigl \langle
P_s[f](\zeta)x, y^* \bigr \rangle &=\int_{\mathbb T} \bigl \langle
x, P_{\zeta}(z)f^*(z)y^* \bigr
\rangle dm(z)\\
& =\Bigl \langle x,  (p)\hskip -.1cm-\hskip -.2cm \int_{\mathbb
T}P_{\zeta}(z) f^*(z)y^* dm(z)\Bigr \rangle,
\endaligned
$$
which gives the result.
\end{proof}

\medskip

We now have:

\begin{thm}\label{gbfyjnlhdsas42}
Let $X$ be a reflexive Banach space and $1\le p\le \infty$. \ If $f \in
H^p_{sot}(\mathbb T, \mathcal B(X, Y))$, then $f^* \in
H^p_{sot}(\mathbb T, \mathcal B(Y^*, X^*))$. \ Moreover,
$P_s[f^*]=P_s[f]^*$.
\end{thm}

\begin{proof} Let $X$ be reflexive, $1\leq p\leq \infty$,
and $f \in H^p_{sot}(\mathbb T, \mathcal B(X, Y))$. \ Since $\bigl
\langle x, P_s[f]^*(\zeta)y^* \bigr \rangle=\bigl \langle
P_s[f](\zeta)x, y^* \bigr \rangle$ for all $x \in X$ and $y^* \in
Y^*$, it follows from Lemma \ref{scxwesrftvsryhik} that $P_s[f]^*\in
\hbox{Hol}\, (\mathbb D, \mathcal B(Y^*, X^*))$. \  For all
$y^*\in Y^*$ and $\zeta \in \mathbb D$,
$$
\aligned ||P_s[f]^*(\zeta)y^*||&=\sup_{||x||=1}\bigl| \bigl \langle
x, P_s[f]^*(\zeta)y^*\bigr \rangle \bigr|\\
&\leq \int_{\mathbb T}P_{\zeta}(z) ||f(z)|| dm(z)\cdot||y^*||\\
&=P[N(f)](\zeta)\cdot||y^*||,
\endaligned
$$
which implies that $||P_s[f]^*(\zeta)||\leq
P[N(f)](\zeta)$. \ It thus follows from Lemma
\ref{sxacdfghnm}(a) that
$$ \int_{\mathbb
T}||\bigl(P_s[f]^*\bigr)_r(z)||dm(z) \leq \int_{\mathbb
T}(P[N(f)])_r(z)dm(z)\leq ||f||_{L^1_{sot}(\mathbb T,\, \mathcal
B(X, Y))}.
$$
This proves that $P_s[f]^*\in H^1(\mathbb D, \mathcal
B(Y^*, X^*))$. \ On the other hand, for all $x \in X$ and $y^* \in
Y^*$, we have that for almost all $z \in \mathbb T$,
\begin{equation}\label{asdfghgfdsa}
\aligned \lim_{r \to 1}\bigl\langle x,  P_s[f]^*(rz)y^* \bigr
\rangle&=\lim_{r \to 1}\langle P_s[f](rz)x,  y^*\rangle\\
&=\lim_{r\to 1}\bigl(P[\langle f(\cdot)x,  y^*\rangle ]\bigr)_r(z)\\
&=\langle x, f^*(z)y^* \rangle,
\endaligned
\end{equation}
where the last equality follows from the fact that $\langle
f(\cdot)x, y^* \rangle \in L^p(\mathbb T)$. \ Since $X^*$ has the
ARNP and $P_s[f]^*(\cdot)y^* \in H^1(\mathbb D, X^*)$, it follows
that
$$
bP_s[f]^*(z)y^*:=\lim_{r \to 1}P_s[f]^*(rz)y^*
$$
exists for almost all $z \in \mathbb T$. \
Since $X$ is reflexive, by the Hahn-Banach Theorem and
(\ref{asdfghgfdsa}), $f^*(\cdot)y^*=bP_s[f]^*(\cdot)y^* \in
H^1(\mathbb T, X^*)$. \ In particular, $f^*$ is SOT measurable, and
hence $f^* \in H^p_{sot}(\mathbb T, \mathcal B(Y^*, X^*))$. \ On the
other hand, since $f \in H^1_{sot}(\mathbb T, \mathcal
B(X,Y))$, it follows from Lemma \ref{gbfyjnlhdsas} that for each
$y^* \in Y^*$ and $\zeta \in \mathbb D$,
$$
P_s[f]^*(\zeta)y^*=(p)\hskip -.1cm-\hskip -.2cm\int_{\mathbb
T}P_{\zeta}(z)f^*(z)y^*dm(z)=P_s[f^*](\zeta)y^*,
$$
which implies $P_s[f^*]=P_s[f]^*$. \ This completes the proof.
\end{proof}

\medskip

Theorem \ref{gbfyjnlhdsas42} may fail if the reflexive condition on
$X$ is dropped. \ To see this, let $f:\mathbb T\to \mathcal B(\ell^1)$
be defined by
$$
(f(z)x)(n):=z^nx(n) \quad(x \equiv (x(n)) \in \ell^1).
$$
Then it is not difficult to show that  $f\in
H^{\infty}_{sot}(\mathbb T, \mathcal B(\ell^1))$ and $f^*$ is not
SOT measurable (cf. Example \ref{ex23}), so that $f^*
\notin H^{\infty}_{sot}(\mathbb T, \mathcal B(\ell^{\infty}))$. \
Note that $\ell^1$ is not reflexive.

\bigskip

%
%
%
%

\section{Appendix: Strong $H^p$-spaces}

\medskip

We devote this section to a general discussion of the circle companions of strong $H^p$-spaces.

For $1\leq p\leq \infty$, let $H^p_{s}(\mathbb D, \mathcal B(X, Y))$
be the space of all functions $h$ in $\hbox{Hol}(\mathbb D, \mathcal
B(X, Y))$ such that $h(\cdot)x \in H^p(\mathbb D, Y)$ for every $x
\in X$: $H^p_{s}(\mathbb D, \mathcal B(X, Y))$ is called a {\it
strong $H^p$-space} (cf. \cite{Ni2}). \ If $h \in H^p_{s}(\mathbb D,
\mathcal B(X, Y))$, then we can easily show that the mapping $x
\mapsto h(\cdot)x$ is a closed linear transformation from $X$ into
$H^p(\mathbb D, Y)$,
so that by the Closed Graph Theorem, it is
bounded. \ Let
$$
||h||_{H^p_s(\mathbb D,\, \mathcal B(X,
Y))}:=\sup\bigl\{||h(\cdot)x||_{H^p(\mathbb D, Y)}: x \in X \
\hbox{with} \ ||x||\leq 1 \bigr\}.
$$
Then $H^p_s(\mathbb D, \mathcal B(X,Y))$ is a normed space and
\begin{equation}\label{301}
H^p(\mathbb D, \mathcal B(X,Y))\subseteq H^p_s(\mathbb D, \mathcal
B(X,Y))\quad (1\leq p\leq \infty).
\end{equation}
Also we can easily check that $H^{\infty}_{s}(\mathbb D, \mathcal
B(X, Y))=H^{\infty}(\mathbb D, \mathcal B(X, Y))$. \ However, if
$1\le p<\infty$ then the inclusion in (\ref{301}) may be proper.

\begin{ex} Let $1\le p<\infty$. \ For  $\zeta \in \mathbb D$,
define $h(\zeta):H^p(\mathbb T) \to \mathbb C$ by
$$
h(\zeta)f:=P[f](\zeta) \quad(f \in H^p(\mathbb T)).
$$
Then for each $\zeta=r e^{i \theta} \in \mathbb D$,
$$
\aligned ||h(\zeta)||_{\mathcal B(H^p(\mathbb T), \mathbb C)}&=
\sup\bigl\{||P[f](re^{i \theta})||: ||f||_{H^p(\mathbb T)}=1\bigr\}\\
&\leq \frac{1+r}{1-r} \cdot\sup\bigl\{||f||_{H^1(\mathbb T)}: ||f||_{H^p(\mathbb T)}=1\bigr\}\\
&\leq \frac{1+r}{1-r},
\endaligned
$$
which implies that $h(\zeta)$ is a bounded linear operator. \ Thus it
is easy to show that $h \in H^p_s(\mathbb D, \mathcal B(H^p(\mathbb
T), \mathbb C))$ and $||h||_{H^p_s(\mathbb D,\,  \mathcal
B(H^p(\mathbb T), \mathbb C))}=1$. \ However, $h \notin H^p(\mathbb D,
\mathcal B(H^p(\mathbb T), \mathbb C))$: indeed, for
each $z \in \mathbb T$, let
$$
f_{z}(s):=e^{\frac{s+z}{s-z}} \quad (s \in \mathbb T).
$$
Then $f_{z}$ is inner, so that $||f_{z}||_{H^p(\mathbb T)}=1$. \ Thus
$$
||h_r(z)||\geq |h_r(z)f_{z}|=e^{\frac{r+1}{r-1}},
$$
so that
$$
\sup_{0\leq r<1}\left(\int_{\mathbb
T}||h_r(z)||^pdm(z)\right)^{\frac{1}{p}}\geq \sup_{0\leq
r<1}e^{\frac{r+1}{r-1}}=\infty,
$$
which implies that $h \notin H^p(\mathbb D, \mathcal
B(H^p(\mathbb T), \mathbb C))$. \qed
\end{ex}

\medskip

Let $\mathcal L(\mathcal X, \mathcal Y)$ be the set of all linear
transformations between normed spaces $\mathcal X$ and $\mathcal Y$.
For a subset $F$ of a Banach space $X$, let $\hbox{sp}(F)$ denote
the linear span of $F$. \ For $1\leq p \leq \infty$, let
$L^p_{s}(\mathbb T, \mathcal L(\hbox{sp}(F), Y))$ be the space of
all (equivalence classes of) functions $f: \mathbb T \rightarrow
\mathcal L(\hbox{sp}(F), Y)$ satisfying
\begin{itemize}
\item[(i)] $f(\cdot)x \in L^p(\mathbb T, Y)$ for
all $x \in \hbox{sp}(F)$; as usual, we identify $f$ and $g$ when
$f(\cdot)x=g(\cdot)x$ in $L^p(\mathbb T, Y)$ for all $x \in
\hbox{sp}(F)$;
\item[(ii)] $||f||_{L^p_{s}(\mathbb T,\, \mathcal L(\hbox{sp}(F),
Y))}:=\sup\bigl\{||f(\cdot)x||_{L^p(\mathbb T, Y)}: x\in
\hbox{sp}(F) \ \hbox{with}\ ||x||\leq1 \bigr\}<\infty$.
\end{itemize}
Then $L^p_{s}(\mathbb T, \mathcal L(\hbox{sp}(F), Y))$ is a normed
space and
$$
L^q_{s}(\mathbb T, \mathcal L(\hbox{sp}(F), Y))\subseteq
L^p_{s}(\mathbb T, \mathcal L(\hbox{sp}(F), Y)) \quad \hbox{if} \
1\leq p \leq q \leq \infty.
$$
We now define $H^p_{s}(\mathbb T, \mathcal L(\hbox{sp}(F), Y))$ as
the space of all (equivalence classes of) functions $f\in
L^p_{s}(\mathbb T, \mathcal L(\hbox{sp}(F), Y))$ such that
$f(\cdot)x\in H^p(\mathbb T,Y)$ for all $x\in\hbox{sp}(F)$.
We note that Definition \ref{dfecfregvryhb} is still well-defined
for functions in $H^1_{s}(\mathbb T, \mathcal L(\hbox{sp}(F), Y))$;
i.e., for $f\in H^1_{s}(\mathbb T, \mathcal L(\hbox{sp}(F), Y))$ and
$x \in \hbox{sp}(F)$,
$$
P_s[f](\zeta)x:=P[f(\cdot)x](\zeta) \quad (\zeta \in \mathbb D).
$$

 We then have:

\begin{lem}\label{sxcvbnbvcv}
Let $X,Y$ be Banach spaces and $F\subseteq X$. \ Suppose
$f \in H^1_{s}(\mathbb T, \mathcal L(\hbox{sp}(F), Y))$.
If $(x_n)$ is a Cauchy sequence in $\hbox{sp}(F)$, then the sequence
$(P_s[f](\cdot)x_n)$ converges uniformly  on every compact subset of
$\mathbb D$.
\end{lem}

\begin{proof}
Suppose $f \in H^1_{s}(\mathbb T, \mathcal L(\hbox{sp}(F), Y))$
and $K$ is a compact subset of $\mathbb D$. \ Then $r\equiv
\max\bigl\{|\zeta|: \zeta \in K \bigr\}<1$. \ Let
$(x_n)$ be a Cauchy sequence in $\hbox{sp}(F)$ and
$\epsilon>0$ be arbitrary. \ For all $\zeta=re^{i \theta} \in K$,
there exists $N>0$ such that if $m>n>N$, then
\begin{equation}\label{sdfagsgsall}
\aligned ||P_s[f](\zeta)x_n-P_s[f](\zeta)x_m||
&=\biggl|\biggl|\int_{\mathbb T}P_{\zeta}(z)f(z)(x_n-x_m)dm(z) \biggr|\biggr|\\
&\leq \frac{1+r}{1-r}\cdot\int_{\mathbb T}||f(z)(x_n-x_m)||dm(z)\\
&\leq \frac{1+r}{1-r} \cdot||f||_{L^1_s(\mathbb T,\,  \mathcal
L(\hbox{sp}(F), Y))}||x_n-x_m||<\frac{\epsilon}{2}.
\endaligned
\end{equation}
Thus $(P_s[f](\cdot)x_n)$ converges pointwise to a function
$h:\mathbb D\to Y$. \ Fixing $n>N$ and letting $m\to \infty$,
(\ref{sdfagsgsall}) leads to
$$
||P_s[f](\zeta)x_n-h(\zeta)||=\lim_{m \to
\infty}||P_s[f](\zeta)x_n-P_s[f](\zeta)x_m||<\epsilon \quad
\hbox{for all} \ \zeta \in K,
$$
which implies $(P_s[f](\cdot)x_n)$ converges uniformly on $K$.
\end{proof}

\medskip

Now if $X$ is separable Banach space, we may define $P_s[f](\zeta)$
on $X$ for all $\zeta \in \mathbb D$ by virtue of Lemma
\ref{sxcvbnbvcv}. \ This is a reason why we introduce $\hbox{sp}(F)$.
Indeed, let $X,Y$ be Banach spaces and assume that $X$ is separable
and $F$ is a dense subset of $X$. \ Then by Lemma \ref{sxcvbnbvcv},
given a function $f \in H^1_{s}(\mathbb T, \mathcal L(\hbox{sp}(F),
Y))$, we may define an extension $\overline P_s[f](\zeta)$ of
$P_s[f](\zeta)$ to $X$ for each $\zeta \in \mathbb D$: in other
words, if $x \in X$, then there exists a sequence $(x_n)$ in
$\hbox{sp}(F)$ such that $x_n \to x$, so that by Lemma
\ref{sxcvbnbvcv}, $(P_s[f](\zeta)x_n)$ is a convergent sequence for
each $\zeta \in \mathbb D$ and hence,
 we can define, for
each $x\in X$,
\begin{equation}\label{3201}
\overline P_s[f](\zeta)x:=\lim_{n \to \infty}P_s[f](\zeta)x_n
\quad(\zeta \in \mathbb D).
\end{equation}
We note that the limit in (\ref{3201}) is independent
of the particular choice of $(x_n)$ because if $(y_n)$ is another
sequence in $\hbox{sp}(F)$ such that $y_n \to x$, then by the same
argument as in (\ref{sdfagsgsall}) we have, for all $\zeta \in
\mathbb D$,
$$
||P_s[f](\zeta)x_n-P_s[f](\zeta)y_n|| \to 0 \quad \hbox{as} \ n \to
\infty,
$$
which implies that the function $\overline P_s[f](\zeta)$ is
well-defined on $X$. \ For simplicity, and since doing
so will not lead to confusion, we will keep denoting by $P_s[f]$ the
extension $\overline P_s[f]$ defined by (\ref{3201}).

\medskip

We then have:

\begin{thm}\label{qqscxwesrftvsryhik}
Let $X,Y$ be Banach spaces and $F$ be a dense subset of
$X$. \ Then the mapping $f \mapsto P_s[f]$ is an isometry from
$H^p_s(\mathbb T, \mathcal L(\hbox{sp}(F), Y))$ to $H^p_s(\mathbb D,
\mathcal B(X, Y))$ for each $1\leq p\leq \infty$.
\end{thm}

\begin{proof} Let $f \in H^p_{s}(\mathbb T, \mathcal L(\hbox{sp}(F),
Y))$ ($1\leq p\leq \infty$) and $\zeta=re^{i \theta} \in \mathbb D$.
Clearly, $P_s[f](\zeta)$ is linear on $X$. \ If
$x \in X$, then there exists a sequence
$(x_n)$ in $\hbox{sp}(F)$ such that $x_n \to x$. \ Thus we have
$$
\begin{aligned}
 ||P_s[f](\zeta)x||&=\lim_{n \to
\infty}\Bigl|\Bigl|\int_{\mathbb T}P_{\zeta}(z)f(z)x_ndm(z)\Bigr|\Bigr|\\
&\leq \frac{1+r}{1-r} \cdot||f||_{L^1_s(\mathbb T,\, \mathcal
L(\hbox{sp}(F), Y))}||x||,
\end{aligned}
$$
which implies that $P_s[f](\zeta) \in \mathcal
B(X, Y)$. \ For each $y^* \in Y^*$, it follows
from Lemma \ref{sxcvbnbvcv} that $\bigl\langle P_s[f](\zeta)x_n, \
y^*\bigr \rangle$ converges uniformly to $\bigl\langle
P_s[f](\zeta)x, \ y^*\bigr \rangle$ on every compact subset of
$\mathbb D$. \ Thus $P_s[f] \in \hbox{Hol}\, (\mathbb D, \mathcal
B(X, Y))$. \ Now we claim that
\begin{equation}\label{dcdvnvbbmmm}
P_s[f] \in  H^p_s(\mathbb D,\, \mathcal B(X,
Y)) \ \hbox{and} \ ||P_s[f]||_{H^p_s(\mathbb D,\, \mathcal
B(X, Y))}\leq ||f||_{L^p_s(\mathbb T,\,
\mathcal L(\hbox{sp}(F), Y))}.
\end{equation}
We split the proof into two cases.

\medskip

Case 1 ($1\le p<\infty$): Let $x$ be an arbitrary unit vector in $X$
and $(x_n)$ be a sequence in $\hbox{sp}(F)$ such that $x_n \to x$. \
Since $P[f(\cdot)x_n] \in H^p(\mathbb D, Y)$, the mapping  $z
\mapsto ||(P_s[f])_r(z)x_n||^p$ is measurable for all $n \in \mathbb
N$. \ Thus it follows from Fatou's lemma and Lemma
\ref{sxacdfghnm}(a) that
$$
\aligned ||(P_s[f])_r(\cdot)x||_{L^p(\mathbb T,
Y)}&=\biggl(\int_{\mathbb
T}\lim_{n \to \infty}||(P_s[f])_r(z)x_n||^pdm(z)\biggr)^{\frac{1}{p}}\\
&\leq \hbox{lim inf}_{n \to \infty}||(P_s[f])_r(\cdot)x_n||_{L^p(\mathbb T, Y)}\\
&\leq \hbox{lim inf}_{n \to \infty}||f(\cdot)x_n||_{L^p(\mathbb T, Y)}\\
&\leq ||f||_{L^p_s(\mathbb T,\, \mathcal L(\hbox{sp}(F), Y))},
\endaligned
$$
which proves (\ref{dcdvnvbbmmm}).

Case 2 ($p=\infty$): Assume to the contrary that
$$
||P_s[f]||_{H^{\infty}_s(\mathbb D,\, \mathcal
B(X, Y))}>||f||_{L^{\infty}_s(\mathbb T,\,
\mathcal L(\hbox{sp}(F), Y))}.
$$
Then there exists a unit vector $x_0$ in $X$
and $\zeta_0 \in \mathbb D$ such that $||P_s[f](\zeta_0)x_0|| \linebreak >
||f||_{L^{\infty}_s(\mathbb T,\, \mathcal L(\hbox{sp}(F), Y))}$.
Choose a sequence $(x_n)$ in $\hbox{sp}(F)$ such that $x_n \to x_0$.
Then for sufficiently large $N$,
$$
||P_s[f](\zeta_0)x_N||> ||f||_{L^{\infty}_s(\mathbb T,\, \mathcal
L(\hbox{sp}(F), Y))}\geq ||f(\cdot)x_N||_{L^{\infty}(\mathbb T, Y)},
$$
which is a contradiction by Lemma \ref{sxacdfghnm}(a). \ This proves
(\ref{dcdvnvbbmmm}) with $p=\infty$.

Now for all $x \in \hbox{sp}(F)$, it follows that
$||f(\cdot)x||_{L^p(\mathbb T, Y)}=||P_s[f](\cdot)x||_{H^p(\mathbb
D, Y)}$ and hence $ ||f||_{L^p_s(\mathbb T,\, \mathcal
L(\hbox{sp}(F), Y))}\leq ||P_s[f]||_{H^p_s(\mathbb D, \mathcal B(X,
Y))}$. \  Therefore, by (\ref{dcdvnvbbmmm}), the
mapping $f \mapsto P_s[f]$ is an isometry from $H^p_s(\mathbb T,
\mathcal L(\hbox{sp}(F), Y))$ to $H^p_s(\mathbb D, \mathcal B(X,
Y))$. \ This completes the proof.
\end{proof}

\begin{cor}\label{sdcrasbghu}
Let $X,Y$ be Banach spaces, and assume that $X$ is separable
and $Y$ has the ARNP. \ For a countable dense subset $F$ of $X$ and
$1\leq p \leq \infty$, the mapping  $f\mapsto P_s[f]$ is an
isometric isomorphism  from $H^p_s(\mathbb T, \mathcal
L(\hbox{sp}(F), Y))$ onto $H^p_s(\mathbb D, \mathcal B(X, Y))$.
\end{cor}

\begin{proof}
A similar argument to (\ref{314003}) shows that the mapping
$f\mapsto P_s[f]$ is a surjection from $H^p_s(\mathbb T, \mathcal
L(\hbox{sp}(F), Y))$ into $H^p_s(\mathbb D, \mathcal B(X, Y))$.
Thus the result follows at once from Theorem \ref{qqscxwesrftvsryhik}.
\end{proof}

\medskip

For $1\leq p \leq \infty$, let $L^p_{s}(\mathbb T,
\mathcal B(X, Y))$ be the space of all (equivalence classes of) SOT
measurable functions $f: \mathbb T \rightarrow \mathcal B(X, Y)$
such that $f(\cdot)x\in L^p(\mathbb T, Y)$ for all $x \in X$; we
identify $f$ and $g$ when $f(\cdot)x=g(\cdot)x$ in $L^p(\mathbb T,
Y)$ for all $x \in X$. \ If $f \in L^p_{s}(\mathbb T, \mathcal B(X,
Y))$, then it follows from the Closed Graph Theorem that
$$
||f||_{L^p_{s}(\mathbb T, \mathcal B(X, Y))}:=\sup
\bigl\{||f(\cdot)x||_{L^p(\mathbb T, Y)}: x \in X \ \hbox{with} \
||x||\leq 1 \bigr\}<\infty.
$$
Then  $L^p_{s}(\mathbb T, \mathcal B(X, Y))$ is a normed space (cf.
\cite{CHL2}, \cite{TJML}, \cite{Pe}). \ In general, the space
$L^p_s(\mathbb T, \mathcal B(X,Y))$ is not complete (cf.
\cite[p.64]{TJML}). \ Also we define $H^p_{s}(\mathbb T, \mathcal
B(X, Y))$ by the space of all (equivalence classes of) functions $f
\in L^p_{s}(\mathbb T, \mathcal B(X, Y))$ such that $f(\cdot)x \in
H^p(\mathbb T, Y)$ for every $x \in X$.

\begin{ex}\label{lsss}
In view of Lemma \ref{sdsacdny}, we may ask whether or not
$L^p_{s}(\mathbb T, \mathcal B(X, Y))= L^p_{sot}(\mathbb T, \mathcal
B(X, Y))$ if $X$ and $Y$ are separable Banach spaces. \ The answer,
however, is negative. \ To see this,  we use the notation
$$
({\bf 1}_F \otimes x)(z):={\bf 1}_F(z) x \quad \hbox{for} \ x \in X,
\ F\subseteq \mathbb T.
$$
Write $H^2\equiv H^2(\mathbb T)$ and define a function $f:\mathbb
T\to\mathcal B(H^2)$ by
\begin{equation}\label{axsdcfvbghn}
(f(z)x)(s):=\widehat{x}(0)+\sum_{n=1}^{\infty}\bigl({\mathbf
1}_{F_n}\otimes \sqrt{2n} \, \widehat{x}(n) \bigr)(z)s^n \quad (x
\in H^2),
\end{equation}
where $F_n:=\{e^{i \theta} : (2-\frac{1}{n})\pi \leq \theta <2\pi
\}$ for $n=1,2,\cdots$. \ Let
$$
E_0:= \{e^{i \theta}: 0 < \theta <\pi \} \quad  \hbox{and} \quad
E_n:=F_n \setminus F_{n+1} \ \  \hbox{for} \ n =1,2,\cdots.
$$
Then for each $1\neq z \in \mathbb T$, there exists
$N\geq 0$ such that $z \in E_N$. \ Thus, by
(\ref{axsdcfvbghn}), we have that for $N\ge 1$,
\begin{equation}\label{kbklklkl}
(f(z)x)(s)=\widehat{x}(0)+\sum_{n=1}^{N}\sqrt{2n}\,\widehat{x}(n)
s^n \quad (z \in E_N),
\end{equation}
which implies $f(z)\in \mathcal B(H^2)$. \ For $x \in H^2$, it
follows from (\ref{kbklklkl}) that
$$
\aligned ||f(\cdot )x||_{L^2(\mathbb T,\,H^2)}^2&=\sum_{N=0}^{\infty}\int_{E_N}||f(z)x||^2dm(z)\\
&=\frac{1}{2}\bigl|\widehat{x}(0)\bigr|^2+
\sum_{N=1}^{\infty}\biggl(\bigl|\widehat{x}(0)\bigr|^2+
\sum_{n=1}^{N}2n\bigl|
\widehat{x}(n)\bigr|^2 \biggr)m(E_N)\\
&=\bigl|\widehat{x}(0)\bigr|^2+\sum_{n=1}^{\infty}2n\bigl|\widehat{x}(n)\bigr|^2m(F_n)\\
&=||x||_{H^2}^2,
\endaligned
$$
which proves that $f\in L^2_{s}(\mathbb T, \mathcal
B(H^2))$.\ However, we have $f\notin
L^2_{sot}(\mathbb T, \mathcal B(H^2))$: indeed if $z \in E_n \
(n\geq 1)$ then it follows from (\ref{kbklklkl}) that
$||f(z)s^n||_{H^2}=\sqrt{2n}$. \ Thus
$$
\int_{\mathbb T}||f(z)||^2dm(z)=\sum_{n=0}^{\infty}
\int_{E_n}||f(z)||^2dm(z) \geq \frac{1}{2}+
\sum_{n=1}^{\infty}\frac{1}{n+1}=\infty,
$$
which implies that $f \notin L^2_{sot}(\mathbb T, \mathcal
B(H^2))$.\qed
\end{ex}

The following diagram summarizes the preceding arguments.

\medskip

Let $X$ be a separable Banach space, $Y$ be a Banach space satisfying
the
ARNP and $F$ be a countable dense subset of $X$. \ Then for all $1\leq
p\leq \infty$,

\begin{alignat}{3} \label{diagram}
H^p(\mathbb T, \mathcal B(X,Y))\subsetneqq\,  & H^p_{sot}(\mathbb T,
\mathcal B(X,Y))\subsetneqq H^p_{s}(\mathbb T, \mathcal B(X,Y))&\,
\subsetneqq\, &H^p_{s}(\mathbb T,\mathcal L(\hbox{sp}(F),Y)) \nonumber \\
& \qquad\rotatebox{90}{$\cong$} \; \bigg\downarrow \; P_s & \ &\qquad\rotatebox{90}{$\cong$} \; \bigg\downarrow \; P_s &\\
&H^p(\mathbb D, \mathcal B(X,Y))   &\   & H^p_{s}(\mathbb D,\mathcal B(X,Y)). \nonumber
\end{alignat}

\medskip
We note that all the inclusions on the first line of (\ref{diagram})
are strict. \ Indeed, in Example \ref{hsot}, we saw
$H^p(\mathbb T,\mathcal B(X,Y))\ne H^p_{sot}(\mathbb T, \mathcal
B(X,Y))$ in general. \ We give some examples such that the other two
inclusions are strict.

\medskip

\begin{ex}\label{hasszdSAXt}
(a) In general, $H^p_{sot}(\mathbb T,\mathcal B(X,Y))\ne
H^p_{s}(\mathbb T, \mathcal B(X,Y))$. \ Indeed, for a separable
Hilbert space $E$, there exists a function $g\in H^2_{s}(\mathbb T,
\mathcal B(E))$ such that $g\notin H^2_{sot}(\mathbb T, \mathcal
B(E))$. \ To see this, we introduce some notations. \ For $f\in
L^2_s(\mathbb T, \mathcal B(E))$, we denote by $f_-$
and $f_+$ the functions
$$
\aligned &f_-(z)x:=\bigl(P_-(f(\cdot)x)\bigr)(\overline{z}) \quad
(z \in \mathbb T,  \ x \in E);\\
&f_+(z) x:=\bigl(P_+(f(\cdot)x)\bigr)(z) \quad  (z \in \mathbb T, \
x \in E),
\endaligned
$$
where $P_+$ and $P_-$ are the orthogonal projections from
$L^2(\mathbb T, E)$ onto $H^2(\mathbb T, E)$ and $L^2(\mathbb T, E)
\ominus H^2(\mathbb T, E)$, respectively (cf.
\cite{Pe}). \ Then, $f_-, f_+\in H^2_s(\mathbb T,
\mathcal B(E))$ and we may write
$$
f(z)= f_+(z)+f_-(\overline{z}) \quad(z \in \mathbb T).
$$
Let $f$ be the function given in (\ref{axsdcfvbghn}). \ Assume that
$f_+ \in H^2_{sot}(\mathbb T, \mathcal B(H^2))$ and $f_- \in
H^2_{sot}(\mathbb T, \mathcal B(H^2))$. \ Observe that for $z \in
\mathbb T$,
$$
||f(z)||^2 \leq
||f_-(\overline{z})||^2+||f_+(z)||^2+2||f_-(\overline{z})||||f_+(z)||.
$$
It thus follows from H\" older's inequality that
$$
\aligned \int_{\mathbb T} ||f(z)||^2dm(z)&\leq \int_{\mathbb
T}||f_-(z)||^2dm(z)+\int_{\mathbb T}||f_+(z)||^2dm(z)\\
&+2\biggl(\int_{\mathbb
T}||f_-(z)||^2dm(z)\biggr)^{\frac{1}{2}}\biggl(\int_{\mathbb
T}||f_+(z)||^2dm(z)\biggr)^{\frac{1}{2}}\\
&<\infty,
\endaligned
$$
which is a contradiction. \ We thus have $f_+ \notin
H^2_{sot}(\mathbb T, \mathcal B(H^2))$ or $f_-\notin
H^2_{sot}(\mathbb T, \mathcal B(H^2))$. \ Note that $H^2$ is a
separable Hilbert space.

\medskip

(b) In general, $H^p_s(\mathbb T, \mathcal B(X,Y))\ne H^p_s(\mathbb
T, \mathcal L(\hbox{sp}(F), Y))$. \ To see this, define $f:\mathbb
T\to \mathcal L(\ell^2,\mathbb C)$ by
$$
f(z)x:=\sum_{n=1}^{\infty}x(n)z^{n} \quad\quad (x \equiv (x(n)) \in \ell^2).
$$
Then $f(z)$ is not bounded for all $z\in\mathbb T$ because for any
$z_0\in\mathbb T$, if we let
$$
x_0(n):=\frac{\overline z_0^n}{n}  \quad (n=1,2,\cdots),
$$
then $f(z_0)x_0=\sum_{n=1}^\infty \frac{1}{n}=\infty$. \ Thus,
$f\notin H^2_s(\mathbb T, \mathcal B(\ell^2,\mathbb C))$. \ On the
other hand, let
$$
F:=\Bigl\{\sum_{n\in\Omega}\alpha_n e_n: \alpha_n \in \mathbb Q \
 \hbox{and $\Omega$ is a finite subset of $\mathbb N$}
\Bigr\},
$$
where $\mathbb Q$ is a countable dense subset of $\mathbb C$ and
$\{e_n: n=1,2,\cdots\}$ is the canonical orthonormal basis for
$\ell^2$. \ Then $F$ is a countable dense subset of $\ell^2$ and we
can easily see that $f \in H^2_{s}(\mathbb T, \mathcal
L(\hbox{sp}(F), \mathbb C))$.
\end{ex}


\vskip 1cm

\noindent {\bf Declarations of interest}: none.

\vskip 1cm

\noindent {\bf Author contributions}: All authors have contributed equally to all relevant areas of the manuscript.

\vskip 1cm

\noindent \textit{Acknowledgments}. \ The authors are deeply indebted to the referee of a previous version for many suggestions on the structure, substance, and style of the paper, which have helped improved the presentation; those suggestions have been incorporated into this version. \ The work of the second named author was
supported by NRF(Korea) grant No. 2022R1A2C1010830. \ The work of the
third named author was supported by NRF(Korea) grant No.
2020R1I1A1A01053085. \ The work of the fourth named author was
supported by NRF(Korea) grant No. 2021R1A2C1005428.


%
%


\vskip 1cm

Ra{\'u}l\ E.\ Curto

Department of Mathematics, University of Iowa, Iowa City, IA 52242,
U.S.A.

E-mail: raul-curto@uiowa.edu

\bigskip

In Sung Hwang

Department of Mathematics, Sungkyunkwan University, Suwon 16419,
Korea

E-mail: ihwang@skku.edu

\bigskip

Sumin Kim

Department of Mathematics, Sungkyunkwan University, Suwon 16419,
Korea

E-mail: suminkim1023@gmail.com

\bigskip

Woo Young Lee

Department of Mathematics and RIM, Seoul National University, Seoul
08826, Korea

E-mail: wylee@snu.ac.kr

\end{document}